\documentclass[runningheads]{llncs}
\usepackage{graphicx}
\usepackage{bm}
\usepackage{booktabs,amsfonts,amsmath,subcaption,xcolor}
\usepackage{blkarray}
\usepackage[ruled,vlined,linesnumbered]{algorithm2e}
\newtheorem{thm}{Theorem}

\newtheorem{prop}[thm]{Proposition}
\newtheorem{cor}[thm]{Corollary}

\newtheorem{defn}{Definition}

\newtheorem{exmp}{Example}

\newtheorem{rmk}{Remark}

\def \bx {\mathbf{x}}
\def \bc {\mathbf{c}}
\def \bd {\mathbf{d}}
\def \by {\mathbf{y}}
\def \bh {\mathbf{h}}
\def \btheta {\mathbf{\theta}}
\def \ba {\mathbf{a}}
\def \bchi {\chi}
\def \bz {\mathbf{z}}
\def \bpi {\pi}
\def \bmu {\mu}
\def \bzero {\mathbf{0}}
\def \bu {\mathbf{u}}
\def \bs {\mathbf{s}}

\def \bR {\mathbb{R}}

\def \conv {\operatorname{conv}}
\def \proj {\operatorname{proj}}
\def \hEI {E_I^{Q}}
\def \hEIchi {E_I^{Q}(\bchi)}

\def \nuIchi{\nu_I^Q(\bchi)}
\def \nuIfn{\nu_{I}^Q(\cdot)}
\def \bnuIab {\bar{\nu}_I^R(\chi;\alpha,\beta)}

\def \gx {g^Q_{\hat{\bx}}}
\newcommand{\gxq}[1]{g^{#1}_{\hat{\bx}}}
\def \gI {g_{\hat{\bx}}^Q(I)}
\def \DRI {D^R_I}
\def \DRIchi {D^R_I(\bchi)}
\newcommand{\exclude}[1]{}

\begin{document}
%
\title{Sparse multi-term disjunctive cuts for the epigraph of a function of binary variables\thanks{Parts of the paper have been published in proceedings of the 23rd International Conference on Integer Programming and Combinatorial Optimization, IPCO 2022. This research is supported by the Office of Naval Research under grant N00014-21-1-2574 and by NSF under grant 2000986.}}
\titlerunning{Sparse disjunctive cuts}
%
%
\author{Rui Chen\inst{1} \and
James Luedtke\inst{2}}
\authorrunning{R. Chen and J. Luedtke}
%
\institute{
Cornell Tech, New York, NY, USA. 
\email{rui.chen@cornell.edu}
\and
University of Wisconsin-Madison, Madison WI, USA.
\email{jim.luedtke@wisc.edu} }
\maketitle              
\begin{abstract}
We propose a new method for separating valid inequalities for the epigraph of a function of binary
variables. The proposed inequalities are disjunctive cuts defined by disjunctive terms obtained by enumerating a subset
$I$ of the binary variables. We show that by restricting the support of the cut to the same set of variables $I$, a cut
can be obtained by solving a linear program with $2^{|I|}$ constraints. While this limits the size of the set $I$ used to
define the multi-term disjunction, the procedure enables generation of multi-term disjunctive cuts using far more terms than
existing approaches. We present two approaches for choosing the subset of variables.
Experience on
three MILP problems with block diagonal structure using $|I|$ up to size 10 indicates the sparse cuts can
often close nearly as much gap as the multi-term disjunctive cuts without this restriction and in a fraction of the
time. We also find that including these cuts within a cut-and-branch solution method for these MILP problems leads to
significant reductions in solution time or ending optimality gap for instances that were not solved within the time
limit. Finally, we describe how the proposed approach can be adapted to optimally ``tilt'' a given valid inequality by
modifying the coefficients of a sparse subset of the variables.

\keywords{Disjunctive cuts \and Epigraph \and Sparsity \and Valid inequalities}
\end{abstract}
\section{Introduction}
We explore techniques for generating valid inequalities (cuts) for the epigraph $E$ of a function $Q':X
\rightarrow \bR$ over binary variables:\begin{equation}\label{epigraph}
	E=\{(\theta,\bx)\in \bR\times X:\theta\geq Q'(\bx) \},
\end{equation}
where $X\subseteq \{0,1\}^n$. 
%
An important application motivating this study is stochastic mixed-integer programming (SMIP) \cite{ahmed2010two}, or
more generally mixed-integer linear programs (MILPs) with block diagonal structures
of the following form: 
\begin{equation}\label{block_diag_MILP}
	\begin{aligned}
		\min~&\bc^T\bx+\sum_{k=1}^N (\bd^k)^T\by^k\\
		\text{s.t. }&T^k\bx+W^k\by^k=\bh^k,~\by^k\geq \bzero,\quad k\in[N],\\
		&\bx\in X\subseteq\{0,1\}^n.
	\end{aligned}
\end{equation}
In the case of two-stage SMIPs, the binary variables $\bx$ represent first-stage decisions, $N$ is the number of scenarios representing the possible
outcomes, and for each  $k \in [N]:=\{1,\ldots,N\}$, the continuous decision variables $\by^k$ represent recourse
actions taken in response to observing the data $(\bd^k,T^k,W^k,\bh^k)$ in scenario $k$.
A common approach to solving such problems is Benders decomposition, which
solves a reformulation of the form
\begin{equation}\label{Benders_reformulation}
	\min_{\btheta,\bx}\Big\{\bc^T\bx+\sum_{k=1}^N \theta_k:\theta_k\geq Q_k(\bx)\text{ for }k\in[N],~\bx\in X \Big\},
\end{equation}
where  for $k\in[N]$, $\bx\in X$, 
\begin{equation}
	\label{eq:subprob}
	Q_k(\bx)=\min_\by\{(\bd^k)^T\by:T^k\bx+W^k\by=\bh^k,~\by\geq \bzero \} .
\end{equation}
The epigraph of $Q_k$ of the form
\eqref{epigraph} shows up as a substructure in \eqref{Benders_reformulation}. In Benders decomposition, valid
inequalities (Benders cuts \cite{bnnobrs1962partitioning,van1969shaped}) for this epigraph are derived via linear programming (LP) duality, but these are not generally
sufficient to define the convex hull of the epigraph, thus motivating the need to derive stronger valid inequalities for
sets of this form. This topic has been extensively studied both theoretically and computationally; see
\cite{chen2021generating,dey2018analysis,gade2014decomposition,laporte1993integer,ntaimo2013fenchel,rahmaniani2020benders,sen2005c,sen2006decomposition,zhang2014finitely}
as just a sample of the literature. Aside from SMIPs, this epigraph substructure appears in a variety of other 
problems (e.g., \cite{bertsimas2021unified,mendez2014branch,wiegele2007biq}).

We study a technique for generating inequalities for $E$ based on a disjunctive relaxation having many terms, specifically
obtained by enumerating all $2^{|I|}$ feasible values for a subset $I$ of the binary variables. 
Disjunctive programming has been a central tool in MILP since its origin in 1970s
\cite{balas1979disjunctive,balas1998disjunctive}. A disjunction is a union of sets, and if the feasible region of an
MILP is
contained within such a union, inequalities valid for the disjunction are valid for the MILP, and are referred to as disjunctive cuts. 
Most disjunctive cuts used in practice are based on two disjunctive terms, e.g., split cuts 
\cite{cook1990chvatal} and lift-and-project cuts \cite{balas1993lift,balas1996mixed,balas2002lift,ceria1998solving}. 
While there has been significant work on classes of cuts that are derived from multiple-term disjunctions
\cite{andersen2007cutting,basu2010maximal,cornuejols2009facets,dey2010experiments,li2008cook}, the current methods
remain focused on disjunctions with a relatively small number of terms. 
Perregaard and Balas \cite{perregaard2001generating} considered an iterative scheme for generating disjunctive cuts from many terms
(see \S \ref{sec:sparse01}), but the approach remains computationally demanding. 

Our proposal for generating multi-term disjunctive cuts more efficiently is based on restricting the support of the
generated cut to the index set $I$, the same set used to define the disjunctive terms. We refer to such cuts as
$I$-sparse cuts. Our approach is motivated by the desire to generate sparse cuts, which may lead to faster solution 
time of the LP relaxations. Recent studies have 
investigated the theoretical
strength of sparse cuts \cite{dey2015some,dey2015approximating,dey2018analysis}. Our use of sparsity is with
respect to the generated cut, which differentiates it from Fukasawa et al. \cite{FPY18} who empirically show that split
cuts derived from (two-term) split disjunctions defined by a sparse integer vector can close the majority of the split closure gap.

In \S \ref{sec:sparse01} we show that the proposed sparsity restriction enables generating multi-term disjunctive cut by solving a single
subproblem per term, and then solving a single cut-generating LP. Thus, while this remains a computationally
demanding cut generation process, we find empirically that it is feasible to use many more disjunctive terms than have
previously been considered. In \S \ref{sec:selection}, we propose two rules for selecting the support $I$ to
generate $I$-sparse inequalities. In \S \ref{sec:computational}, we present results of a computational study using
the $I$-sparse inequalities based on up to $2^{10}$ disjunctive terms on three test problems. We find that in many cases the $I$-sparse cuts close nearly as much gap as
multi-term disjunctive cuts without the sparsity restriction, and can be generated orders of magnitude faster. 
When incorporated into a Benders branch-and-cut solution method, we find that $I$-sparse cuts lead to faster solution
times or smaller ending gaps on our test instances.
Although we find that sparse cuts often can  close a significant portion of the optimality gap, we expect there are
problems where dense cuts may be needed. Thus, we explore in \S \ref{sec:tilting} how we can use our proposed technique
to optimally ``tilt'' a given (possibly dense) valid inequality by modifying a sparse subset of the coefficients of the
inequality. We make concluding remarks in \S \ref{sec:conc}.

A preliminary version of this work appeared in the conference publication \cite{chenipco2022}. In the this paper we
include proofs of the main results, report new results from experiments dynamically adjusting the cardinality of
the cut-support set and using the proposed cuts to solve problems to optimality within a cut-and-branch method, and derive the
approach for tilting a given valid inequality by modifying a sparse subset of its coefficients. Some notation has been changed to improve readability.


\section{Sparse multi-term 0-1 disjunctive cuts}\label{sec:sparse01}

We study the problem of generating valid inequalities for the epigraph $E$ defined in \eqref{epigraph}.
Without loss of generality, we assume the domain $X$ of the function $Q'$ is full-dimensional. (Otherwise, we can
project out certain variables to make the set full-dimensional after projection.) Let $R(X)$ be a (continuous)
relaxation of $X$ with $R(X)\cap \{0,1\}^n=X$. We assume we have access to an extension of $Q'$ to $R(X)$,
$Q:R(X)\rightarrow \mathbb{R}$, satisfying $Q(\bx) = Q'(\bx)$ for $\bx \in X$.
We require that minimizing $Q$ over $R(X)$ is efficiently solvable.
E.g., this would be the case if $R(X)$ is closed, convex, and equipped with an efficient separation oracle and $Q$ is convex over
$R(X)$ with efficiently computable subgradients. 
We emphasize that we do not expect $Q$ to be the convex envelope of $Q'$ over $X$ (i.e., $\conv(E)$) as our interest is precisely about
identifying valid inequalities to approximate this set. 
In case an efficiently computable exact extension $Q$ is is not readily available, one can use an extension $Q$
that instead satisfies $Q(\bx) \leq Q'(\bx)$ for $\bx \in X$. 
For example, in the case of an SMIP having
integer second-stage decisions, the exact recourse function $Q_k(\bx)$ is nonconvex and expensive to evaluate, in which case
one may use instead use the recourse function defined using an LP relaxation of the recourse problem. The strength of the resulting cuts
will naturally depend on the quality of the relaxation, which could for example be improved using standard
MILP valid inequalities.

The following example provides another illustration of the choice of $Q$. 
\begin{exmp}[\cite{li1994global}]
	Assume $Q':\{0,1\}^n\rightarrow\bR$ is defined by\begin{displaymath}
		Q'(\bx)=\frac{\ba^T\bx+b}{\bc^T\bx+d}
	\end{displaymath}
	with $\ba \in \mathbb{R}_+^n$, $\bc \in \mathbb{R}_+^n$ and $d>0$. The natural continuous extension of $Q'$ to $[0,1]^n$
	is not necessarily convex.  However, it is possible \cite{li1994global} to construct a convex extension of $Q'$ over $R(X)=[0,1]^n$ by
	introducing a new variable $y=1/(\bc^T\bx+d)\geq 0$ and linearization variables $z_{i}=x_iy\geq 0$ for $i\in[n]$, and
	define the function $Q:[0,1]^n\rightarrow\bR$ by 
	\begin{subequations}
		\label{eq:qexdef}
		\begin{align}
			Q(\bx)=\min_{y,\bz}~&\ba^T\bz+by\\
			\text{s.t. }&z_i-y\leq 0,~(c_i+d)z_i\leq x_i,~d(y-z_i)\leq 1-x_i,\quad i\in[n],\\
			&\bc^T\bz+dy=1, \bz \in \mathbb{R}_+^n, y\geq 0
		\end{align}
	\end{subequations}
	for $\bx \in [0,1]^n$.  Then, $Q$ is convex over $[0,1]^n$ (because it is the value function of a linear
	program with $\bx$ in the right-hand side of the constraints) and $Q(\bx) = Q'(\bx)$ for $\bx \in \{0,1\}^n$, and
	thus it is a convex extension of $Q'$. Finally, observe that evaluating $Q$ and obtaining a subgradient of
	$Q$ at a point $\bx$ can be accomplished by solving the linear program \eqref{eq:qexdef}. 
\end{exmp}

Let $E^{Q}:=\{(\theta,\bx)\in\bR\times R(X):\theta\geq Q(\bx) \}$ denote the epigraph of $Q$ over $R(X)$ and let $I$ be a nonempty
subset of $[n]$. 
We denote by $\bx_I$ the subvector of $\bx$ with indices $I$, and define
$\{0,1\}^I:=\{ \bx_I : x_i \in \{0,1\}, i \in I\}$.
For each given $\bchi \in \{0,1\}^I$, we define
\[ \hEIchi:=\{(\theta,\bx)\in E^{Q} :\bx_I=\bchi \} = \{ (\theta,\bx)\in\bR\times R(X):\theta\geq Q(\bx), \bx_I = \bchi \}. \]

We derive valid inequalities for $E$ by finding valid inequalities for the following multi-term disjunctive relaxation of $E$:
\begin{equation}\label{def:hEI}
	\hEI:=\bigcup_{\bchi\in \{0,1\}^I}\hEIchi.
\end{equation}
Since  $E \subseteq \hEI$, any inequality valid for $\hEI$ is also valid for $E$.
We call the relaxation $\hEI$ of $E$ a multi-term 0-1 disjunction, and any cut valid for $\hEI$ a
multi-term 0-1 disjunctive cut. 
We include $Q$ as a superscript in the notation $\hEIchi$ and $\hEI$ to emphasize that these relaxations depend
on the choice of the extension $Q$. These relaxations also depend on the choice of $R(X)$, i.e., the domain of $Q$, but we suppress this
dependence for notational convenience.
%

\subsection{Generating multi-term 0-1 disjunctive cuts}
By \eqref{def:hEI}, an inequality of the form $\pi_0\theta+\bpi^T \bx\geq\eta$ is valid for $\hEI$ if and only if\begin{equation}\label{condt:validity}
	\min_{\theta,\bx}\Big\{\pi_0\theta+\bpi^T\bx:(\theta,\bx)\in \hEIchi\Big\}\geq \eta\text{ for all }\bchi\in\{0,1\}^I.
\end{equation}
Therefore, to separate a point $(\hat{\theta},\hat{\bx})$ from $\hEI$, in principle one can solve the following problem:
\begin{subequations}\label{mdcLP}
	\begin{align}
		\min_{\pi_0,\bpi,\eta}~ & \pi_0\hat{\theta}+\bpi^T\hat{\bx}-\eta\\
		\text{s.t. } & \pi_0\theta+\bpi^T \bx \geq \eta,~\forall (\bx,\theta)\in\hEIchi, \bchi\in\{0,1\}^I, \label{mdcLP:rows}\\
		&\pi_0\geq 0, \|(\pi_0,\bpi)\|_1\leq 1, \label{mdcLP:norm}
	\end{align}
\end{subequations}
where \eqref{mdcLP:norm} is just one example of a normalization constraint that can be used to ensure the separation
problem has an optimal solution.

\begin{algorithm}[tb!]
	\SetAlgoLined
	\textbf{Input: }$I \subseteq [n]$\\
	\textbf{Output: }A valid inequality $\hat{\pi}_0\theta+\hat{\bpi}^T \bx\geq\hat{\eta}$ for $\hEI$\\
	Initialize a set $\hat{S}_{\bchi}$ as a subset of extreme points of $\hEIchi$ for each $\bchi\in\{0,1\}^I$;\\
	\Repeat{$\eta(\hat{\pi}_0,\hat{\bpi};\bchi)\geq\hat{\eta}$ for all $\bchi\in\{0,1\}^I$}{
		Compute an optimal solution $(\hat{\pi}_0,\hat{\bpi},\hat{\eta})$ of the LP:\label{alg:rmaster}
		\begin{displaymath}\label{full_multiterm}
			\begin{aligned}
				\min_{\pi_0,\bpi,\eta}\ &\pi_0\hat{\theta}+\bpi^T\hat{\bx}-\eta\\
				\text{s.t. } &\pi_0\theta+\bpi^T\bx \geq \eta,~~\forall (\theta,\bx)\in \hat{S}_\bchi,~\bchi\in\{0,1\}^I,\\
				&\pi_0\geq 0, \|(\pi_0,\bpi)\|_1\leq 1;
			\end{aligned}
		\end{displaymath}\\
		\For{$\bchi\in\{0,1\}^I$}{Solve\label{alg:rsep}\begin{equation}\label{def:RHS}
				\eta(\hat{\pi}_0,\hat{\bpi};\bchi):=\min_{(\theta,\bx)\in \hEIchi}\hat{\pi}_0\theta+\hat{\bpi}^T\bx;
			\end{equation}\\
			\If{$\eta(\hat{\pi}_0,\hat{\bpi};\bchi)<\hat{\eta}$}{Add an optimal solution $(\theta^*,\bx^*)$ of
				\eqref{def:RHS} into $\hat{S}_\bchi$;}
		}
	}
	\caption{The row generating algorithm for solving \eqref{mdcLP}}
	\label{alg:row_gen}
\end{algorithm}

Perregaard and Balas \cite{perregaard2001generating} suggest an iterative row generating algorithm for generating
multi-term disjunctive cuts. Adapting it to our multi-term 0-1 disjunction leads to Algorithm
\ref{alg:row_gen} for solving \eqref{mdcLP}. Specifically, the method alternates between solving a relaxation of
\eqref{mdcLP} defined by only including constraints \eqref{mdcLP:rows} for a (small) subset of the extreme points of
$\hEIchi$ for each $\bchi \in \{0,1\}^I$ (line \ref{alg:rmaster}), and then solving a subproblem for each $\bchi \in
\{0,1\}^I$ to determine if any of the excluded constraints in \eqref{mdcLP:rows} is violated (line \ref{alg:rsep}) and
adding one such constraint if so.
While this approach is guaranteed to yield a valid inequality
for $\hEI$ that cuts off $(\hat{\theta},\hat{\bx})$ when one exists, it is computationally demanding when the number
of terms is larger than just a few.
In particular, the scalability of the algorithm is limited by the multiplied effect of (a) the size of $\{0,1\}^I$, and
(b) the potential need to solve \eqref{def:RHS} multiple times for each $\bchi \in \{0,1\}^I$. Numerical experiments in
\cite{perregaard2001generating} generate valid inequalities for MILPs using only up to 16 disjunctive terms. In this
work, we propose to restrict attention to cuts supported on $I$, which we find eliminates the effect of (b).

\subsection{$I$-sparse inequalities}

We next explore how restricting the support of the generated cut can be used to accelerate the generation of multi-term 0-1 disjunction cuts for $\hEI$ for a fixed $I$.


\begin{defn}
	Let $I \subseteq [n]$. We say an inequality $\theta\geq \mu^T \bx+\eta$ is an $I$-sparse inequality(/cut) for $E$ if the following two conditions hold:\begin{enumerate}
		\item $\theta\geq \mu^T \bx+\eta$ is valid for $\hEI$;
		\item $\mu_i=0$ for all $i\notin I$.
	\end{enumerate}
\end{defn}

The following proposition characterizes $I$-sparse inequalities. 

\begin{prop}\label{prop:sparse}
	An inequality $\theta\geq \bmu^T \bx+\eta$ with $\mu_i=0$ for all $i\notin I$ is an $I$-sparse inequality for $E$ if and only if\begin{equation}\label{condt2:validity}
		\sum_{i\in I}\mu_i\chi_i+\eta\leq \nuIchi,~\forall \bchi\in\{0,1\}^I,
	\end{equation}
	where for each $\bchi \in \{0,1\}^I$,
	\begin{equation}\label{def:nu_I}
		\nuIchi := \min \{ Q(\bx): \bx \in R(X), \bx_I = \bchi\} .
	\end{equation}
\end{prop}
\begin{proof}
	We only need to show that \eqref{condt:validity} with $\pi_0=1$, $\bpi_I=-\mu_I$ and $\bpi_{[n]\setminus I}=\bzero$ holds if
	and only if \eqref{condt2:validity} holds. This is straightforward by observing that for each $\bchi\in\{0,1\}^I$, 
	\begin{align*}
		\min_{\theta,\bx}\Big\{\theta-&\sum_{i\in I}\mu_ix_i:(\theta,\bx)\in \hEIchi\Big\} \\
		&=\min_{\theta,\bx}\Big\{\theta-\sum_{i\in I}\mu_i\chi_i:(\theta,\bx)\in \hEIchi\Big\} \\
		&=\min_{\theta,\bx}\Big\{\theta:(\theta,\bx)\in \hEIchi\Big\} -\sum_{i\in I}\mu_i\chi_i \\
		&=\min_{\theta,\bx}\Big\{\theta: \theta \geq Q(\bx), \bx \in R(X), x_I = \bchi \Big\} -\sum_{i\in I}\mu_i\chi_i \\
		&=\nuIchi-\sum_{i\in I}\mu_i\chi_i.\quad\qed
	\end{align*}
\end{proof}
Observe that the problem \eqref{def:nu_I} has a similar form as \eqref{def:RHS} which is used
when applying the Perregaard and Balas algorithm \cite{perregaard2001generating}  to solve  \eqref{mdcLP}.
We use $Q$ in the superscript of the notation $\nuIchi$ to continue to emphasize the dependence of this quantity on $Q$.
This quantity also depends on $R(X)$, but we suppress this dependence for notational convenience. 

\begin{rmk}
	Our presentation focuses on the case where $X \subseteq \{0,1\}^n$. However, the result naturally extends to problems
	with bounded integer variables by enumerating all possible combinations of values a subset of the bounded integer variables.
\end{rmk}

The following result provides a condition under which every nontrivial
valid inequality for $E$ with coefficients supported on the index set $I$ is an $I$-sparse inequality.
\begin{cor}\label{cor:supportonI}
	If $X=\{0,1\}^n$, $R(X)=[0,1]^n$ and $Q$ is component-wise monotonically nonincreasing or nondecreasing on $R(X)$,
	then an inequality $\theta\geq \bmu^T \bx+\eta$ with $\mu_i=0$ for all $i\notin I$ is valid for $E$ if and only if it is an $I$-sparse inequality.
\end{cor}
\begin{proof}
	As in the proof of Proposition \ref{prop:sparse}, an inequality $\theta\geq \bmu^T \bx+\eta$ with $\bmu_{[n] \setminus I}=\mathbf{0}$ is valid for $E$ if and only if\begin{displaymath}
		\min\{Q(\bx):\bx_I=\bchi, \bx\in\{0,1\}^n \}\geq \sum_{i \in I} \mu_i\chi_i+\eta,~\forall \bchi\in\{0,1\}^I.
	\end{displaymath}
	The conclusion then follows by observing that, for each $\bchi \in \{0,1\}^I$, 
	\begin{align*}
		\nuIchi &=\min\bigl\{Q(\bx):\bx_I=\bchi, \bx\in[0,1]^n \bigr\}\\
		&=\min\bigl\{Q(\bx):\bx_I=\bchi, \bx\in\{0,1\}^n \bigr\},
	\end{align*}
	where the last equality is due to the fact that a monotone function over a box always attains its minimum at an extreme point.\qed
\end{proof}
When the assumptions of Corollary \ref{cor:supportonI} do not hold, 
Proposition \ref{prop:sparse} still provides a method for separating for $I$-sparse cuts -- we just cannot
assure in this case that $I$-sparse cuts contain all cuts supported only on $I$.

Based on Proposition \ref{prop:sparse}, for a fixed $I$, the separation problem for $I$-sparse inequalities for a 
point $(\hat{\theta},\hat{\bx})$ can be solved by solving the LP 
\begin{equation}\label{CGLP}
	\gI=\max\Big\{\sum_{i\in I}\mu_i\hat{x}_i+\eta:\sum_{i\in I}\mu_i\chi_i+\eta\leq \nuIchi,~\bchi\in\{0,1\}^I \Big\}.
\end{equation}
Specifically, the optimal solution of \eqref{CGLP} defines an inequality that cuts off $(\hat{\theta},\hat{\bx})$ if and
only if $\gI > \hat{\theta}$.
When it is easy to determine whether or not a vector is in $\proj_I(X)$, we can replace $\bchi\in\{0,1\}^I$ in
\eqref{CGLP} with $\chi\in\proj_I(X)$ since $\nuIchi=+\infty$ if $\chi\notin\proj_I(X)$. Since $Q$ is finite valued
in $R(X)$, $\nuIchi \in\bR$ for $\chi\in\proj_I(X)$. When $\hat{\bx}_I\in\conv(\proj_I(X))$, the LP \eqref{CGLP} is
guaranteed to have an optimal solution since $X$ being full-dimensional implies that $\proj_I(X)$ is full-dimensional.
When $\hat{\bx}_I\notin\conv(\proj_I(X))$, $(\hat{\theta},\hat{\bx})$ can be cut off by an inequality separating
$\hat{\bx}_I$ from $\proj_I(X)$. 

The main work to generate an $I$-sparse inequality is evaluating $\nuIchi$ by solving \eqref{def:nu_I} for each $\bchi
\in \{0,1\}^I$, and then solving the LP \eqref{CGLP} once.
Note that \eqref{CGLP} has $|I|+1$ variables in contrast to $n+2$ variables in
the problem \eqref{mdcLP} used in the Perregaard and Balas (PB) \cite{perregaard2001generating} algorithm,
and requires solving at most $2^{|I|}$ subproblems of the form \eqref{def:nu_I}, in contrast to the PB algorithm which
solves $2^{|I|}$ subproblems of this form in multiple iterations until convergence.


\subsection{Accelerating the evaluation of $\nuIfn$}
\label{sec:block_diag}

Evaluating $\nuIchi$ for all $\bchi \in \{0,1\}^I$ is the most significant computational component  of generating an
$I$-sparse inequality. 
We discuss techniques to potentially accelerate this evaluation, focusing on our motivating example of MILPs with
block diagonal structures \eqref{block_diag_MILP}. In this context, assume $R(X)=\{\bx \in \mathbb{R}^n :A\bx\leq
\mathbf{b} \}$ is a polyhedral
relaxation of $X$ and assume bound constraints $0 \leq x_i \leq 1$ are included in $A\bx \leq \mathbf{b}$. For a fixed $k \in [N]$
let $Q_k(\bx)$ be as defined in \eqref{eq:subprob}
and assume $Q_k(\bx)$ is
finite valued for all $\bx\in R(X)$. 
In this case, when generating an $I$-sparse inequality for the set $E_k=\{(\theta_k,\bx)\in \bR\times X:\theta_k \geq
Q_k(\bx) \}$
the evaluation of $\nuIchi$ for $\bchi \in \{0,1\}^I$ can be formulated as the following LP
\begin{equation}
	\label{eq:nurecourse}
	\nuIchi=\min_{\bx,\by}\{(\bd^k)^T\by^k:T^k\bx+W^k\by=\bh^k,~\by\geq \bzero,~A\bx\leq \mathbf{b},~\bx_I=\bchi\}. 
\end{equation}
A first simple idea for accelerating the solution of \eqref{eq:nurecourse} for all $\bchi \in \{0,1\}^I$
is to exploit the possibility to warm-start these LPs (see, e.g., \cite{bertsimas1997introduction} for background). LP solvers like Gurobi
\cite{gurobi} automatically implement a simplex warm start when only variable bounds are
changed in a LP. Thus, solving the sequence of problems \eqref{eq:nurecourse} for $\bchi \in \{0,1\}^I$ by making
changes to variable bounds implied by the constraints $\bx_I = \bchi$ will naturally benefit from these warm-start
capabilities.  This motivates a careful selection of the sequence these problems are solved in. For example, by
following the sequence defined by a Gray code \cite{gray1953pulse}, at most one variable bound will change from one
subproblem to the next. 

We do not explore this in our computational study, but another possibility for reducing the time required for
evaluating $\nuIchi$ is to use a simpler to evaluate lower bound on $Q$. E.g., for MILPs with block
diagonal structure, a lower bound on $\nu_I^{Q_k}(\bchi)$ is obtained by solving a problem of the form:
\[ \nu_I^{\hat{Q}_k}(\bchi) = \min \{ \hat{Q}_k(\bx) : \bx \in R(X), \bx_I = \bchi \} \]
where $\hat{Q}_k$ is the current piecewise-linear convex lower bound of $Q_k$ defined by Benders cuts. These lower
bounds could then be used in \eqref{CGLP} which would yield a valid but potentially weaker inequality. This inequality
could then be improved by exactly evaluating $\nu_I^{Q_k}(\bchi)$  for the $\chi$ that correspond to binding
constraints in \eqref{CGLP}, and then re-solving \eqref{CGLP} with these improved values. 

Finally, we note that after evaluating $\nuIchi$ for $\chi \in \{0,1\}^I$ for a given set $I$ (and potentially 
adding a cut based on solving \eqref{CGLP}), we recommend storing these values for future use. In particular, after
re-solving the LP relaxation after addition of cuts and obtaining a new candidate relaxation solution
$(\hat{\theta},\hat{\bx})$, it may be possible that solving \eqref{CGLP} again {\it for the same} set $I$ can lead to a
new violated inequality. Storing the values $\nuIchi$ for $\chi \in \{0,1\}^I$ avoids needing to re-calculate them, so
that only \eqref{CGLP} needs to be solved to determine if such a violated inequality exists.

\section{Two selection rules for the support $I$}\label{sec:selection}

We now discuss techniques for choosing the set $I$ when generating $I$-sparse cuts. Given a point
$(\hat{\theta},\hat{\bx})$, the goal is to select $I$ in order to maximize the cut violation $\gI$ (defined in
\eqref{CGLP}).
Since the complexity of generating these cuts grows exponentially with $|I|$
we investigate techniques that choose $I$ satisfying $|I| \leq K$ for some fixed (small) integer $K$.  We describe two selection
rules that are derived from two different approximations of $Q$.

\subsection{A greedy rule based on a monotone submodular approximation}\label{subsec:greedy}

The problem of choosing $I$ that maximizes $\gI$  is a set function optimization problem. For notational convenience, we
do not distinguish between a set function and a function with binary variables, i.e., we interchangeably use $f(A)$ for
$f(\chi_{A})$ for all $A\subseteq[n]$ where $\chi_{A}\in\{0,1\}^n$ is the indicator vector of $A$. One particular class of set functions satisfying good theoretical properties is monotone submodular functions \cite{grotschel2012geometric}. 
\begin{defn}
	A function $f:2^{[n]}\rightarrow \bR$ is monotone submodular if it satisfies the following two conditions:\begin{enumerate}
		\item (Monotonicity) If $S\subseteq T\subseteq [n]$, then $f(S)\leq f(T)$;
		\item (Submodularity) If $S\subseteq[n]$, $j,~k\in[n]\setminus S$ and $j\neq k$, then $f(S\cup\{j\})+f(S\cup\{k\})\geq f(S\cup\{j,k\})+f(S)$.
	\end{enumerate}
\end{defn}
Given $\hat{\bx}\in[0,1]^n$, we can show that the cut violation function $\gI$ is monotone submodular in $I$ if $Q'$ is
monotone submodular and its extension $Q$ is component-wise monotonically nondecreasing.

\begin{prop}\label{prop:submodular}
	Assume $X=\{0,1\}^n$, $R(X)=[0,1]^n$, $Q'$ is monotone submodular on $X$, and its extension $Q$  is component-wise
	monotonically nondecreasing on $R(X)$. Then the cut violation function $\gx$ is monotone submodular.
\end{prop}
\begin{proof}
	The monotonicity of $\gx$ is obvious since an $I$-sparse inequality is also $I'$-sparse for any $I\subseteq I'\subseteq [n]$. We only need to show submodularity of $\gx$.
	
	For all $I\subseteq[n]$, let $Q_I:\{0,1\}^I\rightarrow\bR$ be the function with $Q_I(\chi_S)=Q'(S)=Q(\chi_S)$ for all $S\subseteq I$ and
	let $\bar{Q}_I:[0,1]^I\rightarrow\bR$ denote the convex envelope of $Q_I$ on $[0,1]^I$. By component-wise monotonicity of $Q$ on $R(X)$,
	\begin{align*}
		\nuIchi=\min\{Q(\bx):\bx_I=\bchi, \bx\in\{0,1\}^n \}=Q_I(\bchi),~\forall \bchi\in\{0,1\}^I.
	\end{align*}
	Therefore,
	\begin{align}
		\gI=\max_{\bmu,\eta} & \sum_{i\in I}\mu_i\hat{\bx}_i+\eta,\nonumber\\
		\text{s.t. } & \sum_{i\in I}\mu_i\chi_i+\eta\leq Q_I(\bchi),~\forall \bchi\in\{0,1\}^I. \label{con:affine_underest}
	\end{align}
	Then by \cite[Corollary 12.1.1]{rockafellar1970convex}, $\gI=\bar{Q}_I(\hat{\bx}_I)$ as
	\eqref{con:affine_underest} characterizes all affine underestimates of $Q_I$. Note that $Q_I$ is submodular by
	submodularity of $Q$. Then the convex envelope $\bar{Q}_I$ is characterized by the Lov\'{a}sz extension
	\cite{lovasz1983submodular} of $Q_I$. We are now ready to show that $\gx$ is submodular. Let $S\subseteq[n]$ and
	$j,k\in[n]\setminus S$ such that $j\neq k$. Let $i_1,\ldots,i_m$ be a reordering of elements in $S\cup \{j,k\}$ such
	that $\hat{x}_{i_1}\geq\ldots\geq\hat{x}_{i_m}$. Define $S_l:=\{i_1,\ldots,i_l\}$ for $l=1,\ldots,m$. Using Lov\'{a}sz extensions of $Q_{S}$, $Q_{S\cup\{j\}}$, $Q_{S\cup\{k\}}$ and $Q_{S\cup\{j,k\}}$, we have\begin{enumerate}
		\item $\gx(S\cup\{j,k\})=(1-\hat{x}_{i_1})Q(\emptyset)+\sum_{l=1}^{m-1}(\hat{x}_{i_l}-\hat{x}_{i_{l+1}})Q(S_i)+\hat{x}_{i_m}Q(S_{i_m})$;
		\item $\gx(S\cup\{j\})=(1-\hat{x}_{i_1})Q(\emptyset)+\sum_{l=1}^{m-1}(\hat{x}_{i_l}-\hat{x}_{i_{l+1}})Q(S_i\setminus\{k\})+\hat{x}_{i_m}Q(S_{i_m}\setminus\{k\})$;
		\item $\gx(S\cup\{k\})=(1-\hat{x}_{i_1})Q(\emptyset)+\sum_{l=1}^{m-1}(\hat{x}_{i_l}-\hat{x}_{i_{l+1}})Q(S_i\setminus\{j\})+\hat{x}_{i_m}Q(S_{i_m}\setminus\{j\})$;
		\item $\gx(S)=(1-\hat{x}_{i_1})Q(\emptyset)+\sum_{l=1}^{m-1}(\hat{x}_{i_l}-\hat{x}_{i_{l+1}})Q(S_i\setminus\{j,k\})+\hat{x}_{i_m}Q(S_{i_m}\setminus\{j,k\})$.
	\end{enumerate}
	Note that $Q(S_i\setminus\{k\})+Q(S_i\setminus\{j\})\geq Q(S_i)+Q(S_i\setminus\{j,k\})$ for $i=1,\ldots,k$ due to
	submodularity and monotonicity of $Q$. It follows that $\gx(S\cup\{j\})+\gx(S\cup\{k\})\geq \gx(S\cup\{j,k\})+\gx(S)$,
	i.e., $\gx$ is submodular.\qed
\end{proof}

\begin{algorithm}[tb!]
	\SetAlgoLined
	\textbf{Input: }$\hat{\bx},~K$\\
	\textbf{Output: }$I$\\
	Initialize $I\leftarrow \emptyset$\\
	\While{$|I|\leq K$}{
		Evaluate $\gx(I\cup\{i\})$ for each $i \notin I$;\\
		$I\leftarrow I\cup\{i^*\}$ where $i^*\in\arg\max_{i\notin I}\gx(I\cup\{i\})$;
	}
	\caption{Greedy algorithm for choosing $I$}
	\label{alg:greedy}
\end{algorithm}

Although maximizing a monotone submodular function subject to a cardinality constraint is NP-hard
\cite{cornuejols1977exceptional} in general, the well-known greedy algorithm of Nemhauser et al.
\cite{nemhauser1978analysis} attains a $1-1/e$ approximation ratio to this problem.
For maximizing $\gI$ subject to a cardinality constraint $|I|\leq K$, the greedy algorithm is described in Algorithm \ref{alg:greedy}.
However, directly applying a greedy algorithm for choosing $I$ may not be a good choice because (i) 
the assumptions of Proposition \ref{prop:submodular} may not hold, and (ii) the greedy algorithm requires evaluating $\gx$ many times, which is
computationally expensive. Therefore, we seek alternatives to this approach by applying the greedy algorithm to a
different cut violation function $\gxq{\tilde{Q}}$ associated with function $\tilde{Q}:[0,1]^{n}\rightarrow\bR$, whose restriction $\tilde{Q}'$ on $\{0,1\}^I$ is an approximation of the function $Q'$. We choose $\tilde{Q}$ 
such  that $\tilde{Q}'$ is monotone and submodular and the cut violation $\gxq{\tilde{Q}}$ can be evaluated much more efficiently than $\gx$.

\sloppypar{We propose to use $\tilde{Q}$ of the form $\tilde{Q}(\bx) = \max_{i\in[n]}\{a_ix_i+b\}$ 	with
	$0\leq a_1\leq\ldots\leq a_n$ (after complementing and reordering some variables). With this form, $\tilde{Q}$ is
	component-wise nondecreasing on $R(X)$ and $\tilde{Q}'$ (the restriction of $\tilde{Q}$ to $X$) is monotone
	submodular on $X$, and thus the associated approximation $\gxq{\tilde{Q}}$ is monotone submodular. 
	To construct such an approximation $\tilde{Q}$, we use the $I$-sparse inequalities with $I=\{i\}$ for each $i\in[n]$.
	When $I=\{i\}$, the polyhedron defined by \eqref{condt2:validity} has a unique extreme point $\big(\nu^Q_{\{i\}}(1)-\nu^Q_{\{i\}}(0),\nu^Q_{\{i\}}(0)\big)$, which corresponds to a valid inequality of $E$: 
	\begin{equation}\label{ineq:1-d}
		\theta\geq \big(\nu^Q_{\{i\}}(1)-\nu^Q_{\{i\}}(0)\big)x_i+\nu^Q_{\{i\}}(0).
	\end{equation}
	By complementing the variable $x_i \leftarrow 1-x_i$ if necessary, we may assume that $\nu^Q_{\{i\}}(1) \geq
	\nu^Q_{\{i\}}(0)$.  Thus, $Q(\bx) \geq \text{LB}^* := \max_{i\in[n]} \nu^Q_{\{i\}}(0)$ for all $\bx \in \{0,1\}^n$. Therefore, we can strengthen \eqref{ineq:1-d} to be 
	$\theta\geq \big(\tilde{\nu}^Q_{\{i\}}(1)-\text{LB}^*\big)x_i+\text{LB}^*$, where $\tilde{\nu}^Q_{\{i\}}(1) =
	\max\{\nu^Q_{\{i\}}(1),\text{LB}^*\}$ for all $i \in [n]$. 
	We thus obtain inequalities of the form $\theta\geq a_ix_i+b$ for $i\in[n]$ with $a_i \geq 0$, which are valid for $E$
	(modulo the mentioned complementing of the $x_i$ variables as needed). Assuming without loss of generality that $a_1\leq\ldots\leq a_n$, we obtain the desired approximation
	$\tilde{Q}$.}

We next discuss how to use $\gxq{\tilde{Q}}$ to generate a support within 
Algorithm \ref{alg:greedy}. In particular, we discuss how to efficiently evaluate  $\gxq{\tilde{Q}}(I)$ for a subset $I$.
Define $\tilde{Q}_I$: $\{0,1\}^I \rightarrow
\mathbb{R}$  by
\[ \tilde{Q}_I(\bx) = \max_{i \in I} \{ a_i x_i + b \} \quad \text{ for all } \bx \in \{0,1\}^I \]
and let $\bar{Q}_I$ be the convex envelope of $\tilde{Q}_I$ over $[0,1]^I$.
As discussed in the proof of Proposition \ref{prop:submodular}, submodularity of $\gxq{\tilde{Q}}$ implies that 
$\gxq{\tilde{Q}}(I)=\bar{Q}_I(\hat{\bx}_I)$. The convex envelope of $\tilde{Q}_I$ over
$[0,1]^I$ is the convex hull of the set:
\begin{equation}
	\label{epigraph:1d}
	F_I=\{(\theta,\bx)\in\bR\times\{0,1\}^I:\theta\geq a_ix_i+b,~i \in I \}.
\end{equation}
The convex hull of this set has been characterized in \cite{atamturk2000mixed,gunluk2001mixing}.
\begin{thm}[\cite{atamturk2000mixed,gunluk2001mixing}]
	Assume $I=\{1,2,\ldots,d\}$ with $0\leq a_1\leq\ldots\leq a_d$. Then\begin{multline}
		\label{eq:mixing}
		\conv(F_I)=\{(\theta,\bx)\in\bR\times[0,1]^d:\theta\geq a_{i_1}x_{i_1}+\sum_{k=2}^m(a_{i_k}-a_{i_{k-1}})x_{i_k}+b,\\
		\text{for all subsequences }(i_k)_{k=1}^m\text{ of }[d]\text{ such that }1\leq i_1\leq\ldots\leq i_m=d \}.
	\end{multline}
\end{thm}
Thus, for a given $\hat{\bx}$, the problem of evaluating $\bar{Q}_I(\hat{\bx}_I)$ can be posed as
$ \min \{ \theta : (\theta,\hat{\bx}_I) \in \conv(F_I) \}$ 
which is equivalent to finding the inequality in the family of inequalities given in \eqref{eq:mixing} with maximimum right-hand side 
when evaluated at $\hat{\bx}_I$. This, in turn, is equivalent to the separation problem of this class of
inequalities,  which 
can be solved in polynomial time \cite{atamturk2000mixed,gunluk2001mixing}.  
We describe the application of the separation algorithm from \cite{gunluk2001mixing} to this context in Algorithm \ref{alg:separation}. Incorporating
this approach for evaluating $\gxq{\tilde{Q}}(I)$ into the greedy algorithm yields a much quicker method for choosing
$I$ than using the greedy algorithm with exact evaluation of $\gx(I)$. Indeed, the most significant work in this case is
solving the problem \eqref{def:nu_I} with $\chi=0$ and $\chi=1$ for each $i \in [n]$ to obtain the values 
$\nu^Q_{\{i\}}(1)$ and $\nu^Q_{\{i\}}(0)$ for $i \in [n]$, which only needs to be done once for the overall greedy
algorithm.

The proposed approximation may not lead to a good choice of $I$ when
$\tilde{Q}$ is not a good approximation of $Q$, in particular because the approximation $\tilde{Q}$ is based on the
maximum of affine lower bounding functions, each supported by a single variable. Thus, it is natural to consider using
affine functions with more general support to build the lower bounding approximation.
However, unless $P=NP$, the following result and the equivalence between optimization and separation \cite{grotschel1981ellipsoid} indicate that 
the key step of evaluating the convex envelope of the given function would no longer be efficiently solvable
even if the support of the inequalities defining the lower approximation of $Q$ were restricted to just two variables
per inequality.
\begin{prop}
	It is NP-hard to optimize a linear function over\begin{equation}\label{setOfInterest}
		\{(\theta,\bx)\in\bR\times\{0,1\}^n: \theta\geq \ba^T \bx+b,(\ba,b)\in\mathcal{A} \}
	\end{equation}
	even if $\|\ba\|_0\leq 2$ for each $(\ba,b)\in\mathcal{A}$ and $|\mathcal{A}|$ is polynomially bounded by $n$.
\end{prop}
\begin{proof}
	We prove by polynomially reducing an arbitrary instance of the $NP$-complete vertex cover problem to a linear
	optimization problem over \eqref{setOfInterest} with $\|\ba\|_0\leq 2$ and $|\mathcal{A}|=O(n^2)$. The vertex cover problem is stated as:\begin{itemize}
		\item Given an undirected graph $G=(V,E)$ and positive integer $k'$, does there exist $V'\subseteq V$ with $|V'|\leq k'$ such that $u\in V'$ or $v\in V'$ for each $uv\in E$?
	\end{itemize}
	We next show that such vertex cover $V'$ exists if and only if the optimal objective value of the following problem is at most $-1$:\begin{equation}\label{opt:VC}
		\min\Big\{\sum_{v\in V}x_{v}+(k+1)\theta: \theta\geq -1;~\theta\geq -x_{u}-x_{v},uv\in E;~\bx\in\{0,1\}^V \Big\}.
	\end{equation}
	Note that \eqref{opt:VC}$\leq-1$ if and only if the optimal solution $(\theta^*,\bx^*)$ satisfies $\theta^*=-1$,
	$\sum_{v\in V}x^*_{v}\leq k$ and $-1\geq-x^*_{u}-x^*_{v}$ for each $uv\in E$. Such $\bx^*$ corresponds to a vertex cover $V':=\{v\in V:x^*_v=1\}$ of $G$ with $|V'|\leq k'$. On the other hand, a vertex cover $V'$ of $G$ with $|V'|\leq k'$ corresponds to an optimal solution $(-1,x^*)$ of \eqref{opt:VC} with objective value at most $-1$ satisfying $x^*_v=1$ if and only if $v\in V'$.\qed
\end{proof}
\begin{algorithm}[tb!]
	\SetAlgoLined
	\textbf{Input: }$\hat{\bx},~I=\{1,\ldots,d\}, b, a_i, i\in I$ with $0\leq a_1\leq \ldots\leq a_d$\\
	\textbf{Output: }$\gxq{\tilde{Q}}(I)$\\
	Initialize $\hat{x}_{\max}\leftarrow-\infty$, $i_{\max}\leftarrow d$,
	$(\sigma_i)_{i=1}^d\leftarrow(\emptyset)_{i=1}^d$;\\
	\For{$i=d,d-1,\ldots,1$}{
		\If{$\hat{x}_i>\hat{\bx}_{\max}$}{
			$\sigma_i\leftarrow i_{\max}$,
			$\hat{x}_{\max}\leftarrow\hat{x}_i$, $i_{\max}\leftarrow i$;}
	}
	$\tilde{g}\leftarrow b +a_{i_{\max}}\hat{x}_{i_{\max}}$, $k\leftarrow i_{\max}$;\\
	\While{$k\neq d$}{
		$\tilde{g}\leftarrow\tilde{g}+(a_{\sigma_k}-a_{k})\hat{x}_{\sigma_k}$;\\
		$k\leftarrow \sigma_k$;
	}
	Return $\gxq{\tilde{Q}}(I) = \tilde{g}$;
	\caption{Evaluating the violation underestimate $\gxq{\tilde{Q}}(I)$.}
	\label{alg:separation}
\end{algorithm}

\subsection{A cutting-plane approximation rule}\label{subsec:cutpl_rule}

We next describe an alternative selection rule for $I$ that is based on a single affine lower bound (e.g., from a
cutting-plane) of $Q'$. Let $\ba \in \mathbb{R}^n$, $b \in \mathbb{R}$, and \begin{displaymath}
	F_{(\ba,b)}=\{(\theta,\bx)\in\bR\times\{0,1\}^n:\theta\geq \ba^T\bx+b \}.
\end{displaymath}
Let $(\hat{\theta},\hat{\bx})\in\bR\times[0,1]^n$ be given, and consider the problem of finding a valid inequality for $F_{(a,b)}$ of the form
\begin{equation}
	\label{eq:genfab}
	\theta\geq\sum_{i\in I}\mu_ix_i+\eta
\end{equation}
that is maximally violated by $(\hat{\theta},\hat{\bx})$.

\begin{prop}\label{prop:singleCut}
	The problem of maximizing
	\[ \sum_{i \in I} \mu_i \hat{x}_i + \eta - \hat{\theta} \] 
	such that inequality \eqref{eq:genfab} defined by $\mu,\eta$ is valid for $F_{(\ba,b)}$ has optimal value $-\sum_{i=1}^na_i^-+\sum_{i\in I}(a_i\hat{x}_i+a_i^-)+b-\hat{\theta}$,	where $a_i^-=\max\{-a_i,0\}$.
\end{prop}
\begin{proof}
	Inequality $\theta\geq\sum_{i\in I}\mu_ix_i+\eta$ is valid for $F_{(\ba,b)}$ if and only if
	\begin{align*}
		\sum_{i\in I}\mu_i\chi_i+\eta &\leq \min\{\theta:(\theta,\bx)\in F_{(\ba,b)},\bx_I =\bchi \} \\
		&=\sum_{i\in I}a_i\chi_i-\sum_{i\notin I}a_i^-+b,~\forall \bchi\in\{0,1\}^I.
	\end{align*}
	Therefore, by LP duality, the maximum violation of an inequality of this form is
	\begin{align*}
		\max_{\bmu,\eta}&\Big\{\sum_{i\in I}\mu_i\hat{x}_i+\eta-\hat{\theta}:\sum_{i\in I}\mu_i\chi_i+\eta\leq\sum_{i\in
			I}a_i\chi_i-\sum_{i\notin I}a_i^-+b,~\bchi\in\{0,1\}^I \Big\}\\
		&=\min\Big\{\sum_{\chi\in\{0,1\}^I}\sum_{i\in I}a_i\chi_i\lambda_\chi+\sum_{\chi\in\{0,1\}^I}\Big(-\sum_{i\notin
			I}a_i^-+b\Big)\lambda_\chi: \\ 
		& \qquad\qquad \qquad\sum_{\chi\in\{0,1\}^I}\chi_i\lambda_\chi=\hat{x}_i,i\in I;\sum_{\chi\in\{0,1\}^I}\lambda_\chi=1,\lambda\geq 0 \Big\}-\hat{\theta}\\
		&=\sum_{i\in I}a_i\hat{x}_i-\sum_{i\notin I}a_i^-+b-\hat{\theta} \\
		&=-\sum_{i=1}^na_i^-+\sum_{i\in I}(a_i\hat{x}_i+a_i^-)+b-\hat{\theta}.\quad\qed
	\end{align*}
\end{proof}

Using Proposition \ref{prop:singleCut}, we interpret the value $a_i\hat{x}_i+a_i^-$ as a measure of the importance of variable
$x_i$ for the cutting plane $\theta\geq \ba^T \bx+b$ at $\hat{\bx}$. We use this intuition to construct a selection rule. 
We first pick a cutting plane $\theta\geq \ba^T \bx+b$ that approximates the epigraph of $Q$ at $\hat{\bx}$. Then indices $i\in[n]$ are
added to the set $I$ in decreasing order of the value $a_i\hat{x}_i+a_i^-$ until $|I|=K$.  
Note that $a_i\hat{x}_i+a_i^-\geq 0$ for any $a_i\in\bR$ and $\hat{x}_i\in[0,1]$. 
If the cutting plane approximation $\theta\geq \ba^T \bx+b$ is sparse (i.e., $|\{ i \in [n]: a_i \neq 0\}|$ is small), it is possible that $|\{i\in[n]:a_i\hat{x}_i+a_i^->0\}|<K$. In such cases, we first add those indices with positive $a_i\hat{x}_i+a_i^-$ values into $I$, then pick another cutting plane and repeat the procedure until $|I|=K$. 
A potential advantage of this selection rule is that it does not require any evaluation of the cut violation function.
And unlike the selection rule in \S \ref{subsec:greedy}, this selection rule can take advantage
of the availability of dense cutting plane approximations. The potential limitation, of course, is the reliance on the
single cutting-plane approximation.

The final detail we need to specify for this approach is how to choose the cutting-plane approximation(s). Assume a
collection $\mathcal{A}$ of cutting planes of the form $\theta \geq \ba^T \bx + b$ is available. A natural choice for $\mathcal{A}$ is the set of 
cutting planes (e.g., Benders cuts) that have been added in the algorithm so far for approximating $E$. A natural ordering for choosing which cutting plane
in $\mathcal{A}$ to use first is based on the tightness of the cutting plane at the point $\hat{\bx}$. The inequality in
$\mathcal{A}$ with
coefficients $(\ba,b)$ that yield the highest $\ba^T\hat{\bx}+b$ value is chosen first, etc. 

\section{Computational results}\label{sec:computational}
To provide insight into the computational potential of $I$-sparse cuts, we conduct numerical experiments on three MILP problems with block diagonal structures \eqref{block_diag_MILP}:\begin{itemize}
	\item The stochastic network interdiction (SNIP) problem \cite{pan2008minimizing}: $n=320$ for these instances.
	\item The latent-class logit assortment (LLA) problem \cite{mendez2014branch}: $n=500$ for these instances.
	\item A stochastic version of the capacitated facility location (CAP) problem \cite{bodur2017strengthened}: $n$
	ranges between 25 and 50 for these instances.
\end{itemize} 
We present the problem definition and details of the test instances for each problem in the Appendix.
For the first two test problems, each block of their MILP formulations is sparse in variables $\bx$, but in distinct ways.
For the SNIP problem, we observe that when applying Benders decomposition to solve its LP relaxation the Benders cuts
are mostly very sparse in $\bx$. In the LLA problem each block of the MILP formulation only uses a small portion (between
12 and 20) of the $\bx$ variables, making the use of sparse cuts very natural for this problem. Neither of these two sparsity properties 
holds for the CAP problem.

The constraints $\bx\in X$ in all our test problems 
consist of $\bx$ being binary and either a lower-bounding or upper-bounding cardinality constraint on the number of
nonzero $x_i$ variables. Therefore, we use $R(X)=\conv(X)$ for all our tests instances. We use the direct LP
relaxation as $Q_k$ for each block of the MILP as described in \S \ref{sec:block_diag}. 

We test the ability of $I$-sparse cuts to improve upon the standard LP relaxation within the Benders
reformulation \eqref{Benders_reformulation}. 
The cut generating process is described in Algorithm \ref{alg:cut_generation}. 
In the first step (line \ref{step:Benders}), we add standard Benders cuts iteratively until we have solved the initial
LP relaxation. Specifically, this Benders approach works with a master LP relaxation in which the constraints $\theta_k \geq Q_k(\bx), k \in
[N]$ are approximated by Benders cuts of the form:
\[ \theta_k \geq (\bpi^j)^\top (\bh^k - T^k \bx), \quad j=1,\ldots, t_k \]
where $\bpi^j$, $j=1,\ldots,t_k$, are extreme point solutions to the dual feasible region for subproblem
$k$, $\Pi^k = \{ \bpi : \bpi^\top W^k \leq \bd^k\}$. In the standard cutting-plane implementation \cite{kelley1960cutting},
after solving a master LP relaxation and obtaining a solution $(\hat{\btheta},\hat{\bx})$, Benders cuts are identified by solving the
subproblem \eqref{eq:subprob} with $\bx=\hat{\bx}$ and adding the Benders cut defined by the dual optimal solution 
if it is violated by $(\hat{\btheta},\hat{\bx})$. 
We use this standard cutting-plane method for the SNIP and LLA instances. We
found the cutting-plane method took too long to converge for CAP instances, so we use the level method
\cite{lemarechal1995new} for line \ref{step:Benders} of Algorithm \ref{alg:cut_generation} on those instances.

In terms of the $I$-sparse cut generation, we consider the following variants of Algorithm \ref{alg:cut_generation}:
\begin{itemize}
	\item Greedy-$K$: Use the greedy rule described in \S \ref{subsec:greedy} for generating the support $I$ of size $K$;
	\item Cutpl-$K$: Use the cutting plane approximation rule described in \S \ref{subsec:cutpl_rule} for generating the support $I$ of size $K$;
\end{itemize}
We test Greedy-$K$ and Cutpl-$K$ with $K$ fixed at $4$, $7$, and $10$. We also test adaptive variants,
Greedy-Ad and Cutpl-Ad of each selection method. These variants begin with $K=4$. After solving the master LP,
if $K < 10$ and the gap closed in the last five iterations  for this $K$ is less than 1\% of the total gap closed thus
far we increase $K$ by 1 and re-start the generation of $I$-sparse cuts.


For Cutpl, we use the collection of all the Benders cuts added for block $k$ in line \ref{step:Benders} of Algorithm \ref{alg:cut_generation} as $\mathcal{A}$ for $Q_k$.
To improve the efficiency of the algorithm, when applying Greedy, we only select $I$ from indices for which the
corresponding variables have a nonzero coefficient in at least one of the Benders cuts for block $k$. This restriction
is also implicitly implemented when using Cutpl since indices $i$ with $a_i=0$ for all $(\ba,b)\in\mathcal{A}$ can never be selected by Cutpl. It significantly improves the efficiency of Greedy on SNIP instances (by skipping the generation of $\{i\}$-sparse cuts for most $i\in[n]$). 

All LPs and MILPs are solved using Gurobi 9.1.0. 
\begin{algorithm}[tb!]
	\caption{Generating $I$-sparse cuts}
	\label{alg:cut_generation}
	\SetAlgoLined
	Initialize a master LP using Benders decomposition;\label{step:Benders}\\
	\Repeat{No violated cut can be generated or time limit is reached}{
		Solve the master LP to obtain solution $(\hat{\theta},\hat{\bx})$;\\
		\For{$k\in[N]$}{
			Choose a support $I$;\label{step:Choose_I}\\
			Generate an $I$-sparse cut valid for the set $E_k = \{ (\theta_k,\bx) \in \mathbb{R}\times X : \theta_k \geq Q_k(\bx) \}$ by solving \eqref{CGLP};\label{step:Generate_cut}\\
			Add the $I$-sparse cut to the master LP if it is violated by $(\hat{\theta}_k,\hat{\bx})$;\label{step:Add_cut}}}
\end{algorithm}

\begin{figure}[hbt!]
	\centering
	\includegraphics[width=0.8\linewidth]{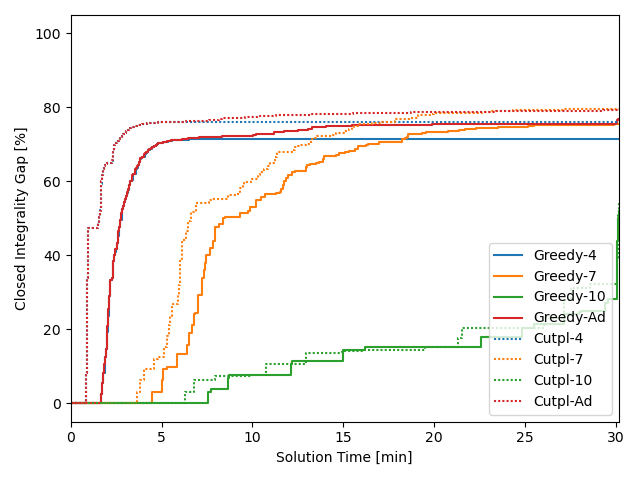}
	\caption{Integrality-gap-closed profiles for SNIP instances obtained by different Greedy rules (solid) and Cutpl rules (dashed)}
	\label{fig:SNIPK}
\end{figure}

\begin{figure}[hbt!]
	\centering
	\includegraphics[width=0.8\linewidth]{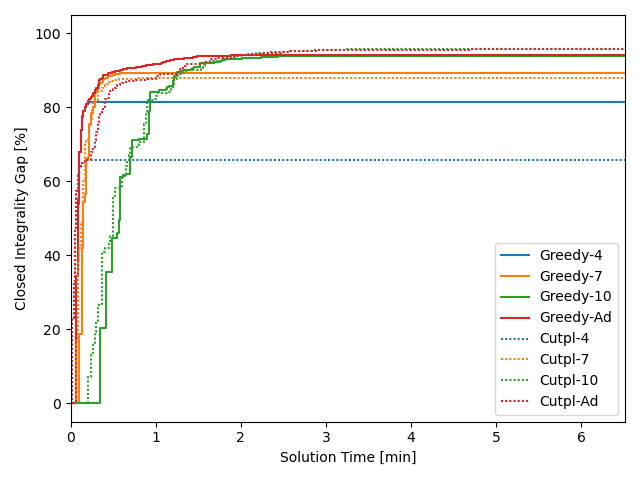}
	\caption{Integrality-gap-closed profiles for LLA instances obtained by different Greedy rules (solid) and Cutpl rules (dashed)}
	\label{fig:LLA_K}
\end{figure}

\begin{figure}[hbt!]
	\centering
	\includegraphics[width=0.8\linewidth]{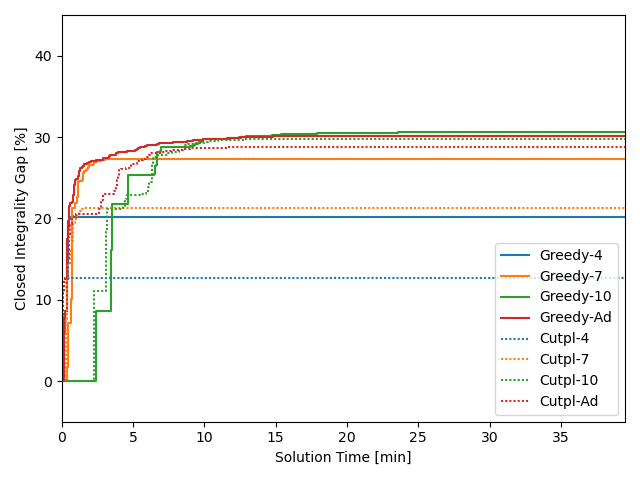}
	\caption{Integrality-gap-closed profiles for CAP instances obtained by different Greedy rules (solid) and Cutpl rules (dashed)}
	\label{fig:CAP_K}
\end{figure}

\subsection{LP relaxation results}

We first present results showing the impact of adding $I$-sparse inequalities to the LP relaxation of the problem without branching.
An 1800-second time limit is set for generating $I$-sparse cuts in these experiments. 
To visually compare the performance of $I$-sparse cuts across multiple test instances, we
present results in the form of an integrality-gap-closed profile. Each curve in such a profile corresponds to a
particular cut generation strategy, and its value at time $t$
represents the average (over the set of instances for that problem class) integrality gap closed by time $t$, where the
integrality gap closed at time $t$ is calculated as $(z_{R}(t) - z_{LP})/(z^* - z_{LP}) \times 100\%$, where $z_R(t)$ is the bound
obtained by the algorithm at time $t$, $z_{LP}$ is the basic LP relaxation bound, and $z^*$ is the optimal value.

The results for the SNIP, LLA, and CAP test problems are given in Figures \ref{fig:SNIPK}, \ref{fig:LLA_K}, and
\ref{fig:CAP_K}, respectively, where in each case we vary $K \in \{4,7,10\}$ or use an adaptively chosen $K$, and compare the Greedy and Cutpl
selection rules. In each case we find that the two different selection rules have similar trends in gap closed over
time. For fixed $K$, Cutpl rules perform better on the SNIP test instances, whereas Greedy rules have significantly better
performance
on the LLA and CAP instances when $K=4$ or $K=7$.
In terms of the effect of $K$, as expected smaller values of $K$ yield quicker initial gap improvement, whereas larger
values of $K$ require more time to close the gap but eventually lead to more gap closed. For the SNIP instances we find that using $K=4$ already closes most of
the gap, and does so much more quickly than with $K=7$ or $K=10$. For the LLA instances we find that increasing $K$
leads to more gap closed, although significant gap is already closed with $K=4$, and the additional gap closed using
$K=10$ is marginal, while requiring significantly more time.  
For the CAP instances, we find that the $I$-sparse cuts close significantly less gap than the other test problems,
although the gap closed is still significant. Large values of $K$ yield significantly more gap closed on the CAP
instances, but also requires considerably longer running time. The adaptive approaches Greedy-Ad and Cutpl-Ad
appear to successfully achieve the best of the different choices of $K$, e.g., yielding quick improvement in bound early
on while also eventually achieving bound improvement as good as achieved with the largest $K$. 

We observe that the number of $I$-sparse cuts added by the algorithm does not increase when $K$ increases. Thus, the improvement in the bound is attributable to stronger cuts rather than an increase in the number of cuts added.

\begin{figure}[hbt!]
	\centering
	\begin{subfigure}[b]{0.49\linewidth}
		\includegraphics[width=\linewidth]{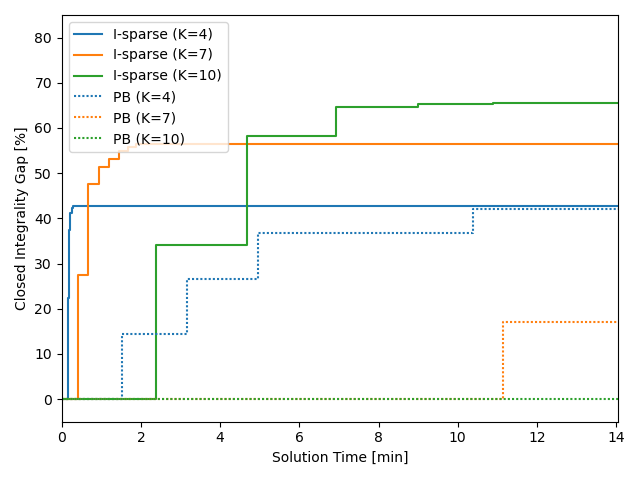}
	\end{subfigure}
	\begin{subfigure}[b]{0.49\linewidth}
		\includegraphics[width=\linewidth]{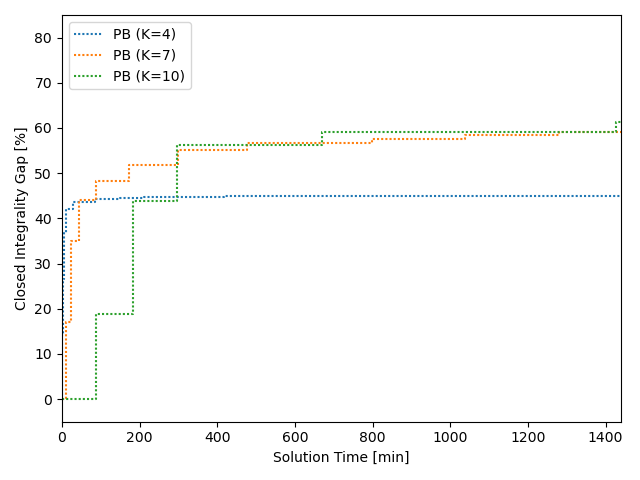}
	\end{subfigure}
	\begin{subfigure}[b]{0.49\linewidth}
		\includegraphics[width=\linewidth]{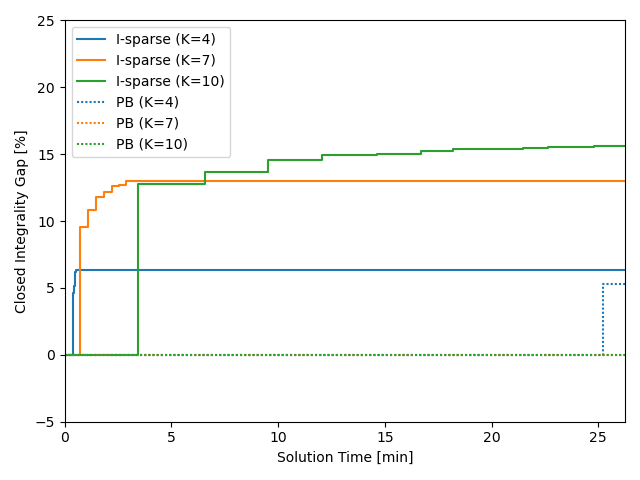}
	\end{subfigure}
	\begin{subfigure}[b]{0.49\linewidth}
		\includegraphics[width=\linewidth]{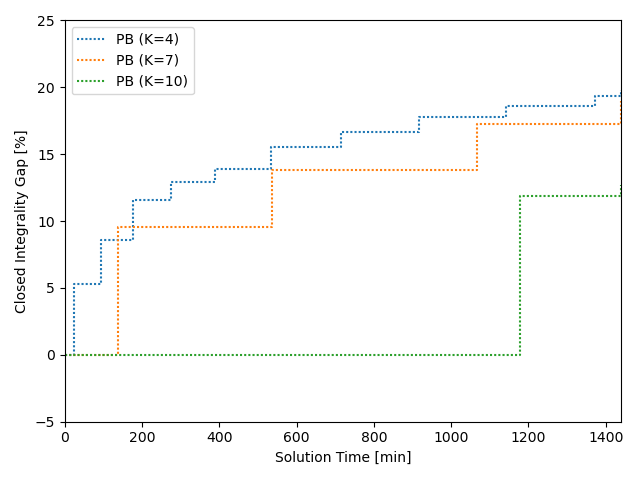}
	\end{subfigure}
	\caption{Integrality gap closed by $I$-sparse cuts and cuts generated by the PB algorithm on instances CAP101 (top) and CAP111 (bottom)}
	\label{fig:IsparseVSPB}
\end{figure}

We next compare the $I$-sparse cuts with the multi-term 0-1 disjunctive cuts without the sparsity
restriction, but generated from the same sets $I$, where the cuts are generated using the Perregaard and Balas (PB)
\cite{perregaard2001generating} approach. Our interest in this comparison is to demonstrate the potential time
reductions from using the $I$-sparse cuts and to estimate the extent to which the sparsity restriction degrades the
quality of the relaxation. We conduct this experiment only on the CAP test instances, since we have already seen that
the $I$-sparse cuts are sufficient to close most of the gap in the SNIP and LLA instances, and thus there is little
potential to close more gap when eliminating the sparsity restriction. 
We
set a 24-hour time limit for the PB algorithm. For both the $I$-sparse and PB cuts, we use Greedy-$K$ as the rule for
selecting the  set $I$ to define the multi-term disjunction. 

Figure \ref{fig:IsparseVSPB} displays the integrality gap closed over time for two specific CAP instances, one for which $I$-sparse cuts were
able to close a significant portion of the gap (CAP101), and one for which they were not (CAP111). The figures on the
left display results for both the $I$-sparse cuts (solid lines) and PB cuts (dashed lines), with the time-scale
($x$-axis) determined by the time required to generate all $I$-sparse cuts for the largest value of $K$. From these
figures we observe that for any value of $K$, within this time frame the $I$-sparse cuts close significantly more gap
than the PB cuts. To estimate the potential for PB cuts to eventually close more gap, we show the  gap closed by the PB
cuts over the full 24-hour time limit in the figures on the right. For CAP101 we find that the PB cuts do not close more
gap than the $I$-sparse cuts, suggesting that the sparsity restriction is not significantly degrading the strength of
the cuts in this case. On the other hand, for CAP111, we find that when given enough time the PB cuts can close
significantly more gap, as seen particularly for the $K=4$ results, although requiring far more time to do so. For both
CAP instances, we observe that most of the generated PB cuts are as sparse as the $I$-sparse cuts in the first few
iterations but become significantly denser (e.g., with non-zeros on more than half the variables) in later iterations.

\subsection{Solving to optimality}
We next present empirical results using $I$-sparse cuts within a branch-and-cut algorithm for exactly solving the test
instances. 
The purpose of this study is to verify that the demonstrated relaxation improvement from these cuts 
translates to a reduction in the size of the search tree for these instances. We emphasize that our purpose is not to attempt to
use these cuts to obtain state-of-the-art results, as such a test would require significant care in integrating multiple
different types of cuts, etc.

We investigate using $I$-sparse cuts added at the root node to obtain an improved LP relaxation of
the MILP \eqref{Benders_reformulation}, leading to a method we refer to as IBC ($I$-sparse branch-and-cut).
We then solve the MILP instance, strengthened with the $I$-sparse cuts, via a Benders branch-and-cut algorithm.
In this method,
{\it Benders} cuts are added as lazy cuts at nodes in the branch-and-bound tree, as needed when integer feasible
solutions are encountered.
Specifically, when a solution $(\hat{\theta},\hat{\bx})$ with $\hat{\bx}$ integer valued is obtained in the search process
(either via a heuristic or as a solution of a node relaxation subproblem) we check whether it is feasible to
\eqref{Benders_reformulation}, i.e., whether $\hat{\theta}_k\geq Q_k(\hat{\bx})$ for all $k\in[N]$. If not we add Benders
cuts as lazy constraints to cut off this infeasible solution and continue branch-and-cut. 
A more detailed description of Benders branch-and-cut can be found, e.g., in
\cite{bodur2017strengthened,chen2021generating}.
We emphasize that we add $I$-sparse cuts only at the root node, so in terms of evaluation of the use of $I$-sparse cuts, this is a cut-and-branch
approach. For generating $I$-sparse cuts, we use the Cutpl-Ad method
from the previous section on the SNIP instances,
and Greedy-Ad on the LLA and CAP instances.  To achieve a balance between the benefit from the gap closed from the
$I$-sparse cuts and the cut generation time, we terminate the cut generation process 
if the gap closed in the first two iterations of adding $I$-sparse cuts for a fixed $K$ is smaller than 1\% of the gap
closed since the beginning of the $I$-sparse cut generation process. This choice is based on the empriical observation that 
the largest gap improvement almost always occurs in the
first two iterations for each fixed $K$.

We compare against two other exact solution approaches, EXT --- solving the MILP \eqref{block_diag_MILP} directly in
extensive form, and BBC --- vanilla Benders branch-and-cut, which is identical to the IBC implementation
except that the $I$-sparse cut generation step is skipped. A 3600-second time limit is set for solving each 
instance (including cut generation). Because the SNIP instances can be easily solved by
BBC when solver cuts are used on top of the Benders cuts \cite{bodur2017strengthened,chen2021generating}, when solving
the SNIP instances using BBC 
or IBC, we turn off Gurobi presolve and cuts to show the impact of the $I$-sparse cuts. No changes to Gurobi's settings are made
for the LLA or CAP instances, or for SNIP instances when solving with EXT.

\begin{table}[tb!]
	\centering
	\caption{Exact solution results for SNIP instances. Time and gap results are averages over five instances.}
	\begin{tabular}{cccccccccc}
		\toprule
		$b$ & \multicolumn{3}{c}{Avg soln time (s)} & \multicolumn{3}{c}{\# solved instances} & \multicolumn{3}{c}{Avg opt gap (\%)}\\
		\cmidrule(lr){2-4}\cmidrule(lr){5-7}\cmidrule(lr){8-10}
		& EXT & BBC & IBC & EXT & BBC & IBC & EXT & BBC & IBC \\
		\midrule\midrule\addlinespace
		30 & $\geq$3600 & \phantom{$\geq$0}184 & 363 & 0/5 & 5/5 & 5/5 & 12.6 & 0.0 & 0.0 \\
		50 & $\geq$3600 & $\geq$1675 & 505 & 0/5 & 4/5 & 5/5 & 19.0 & 0.5 & 0.0 \\
		70 & $\geq$3600 & $\geq$3534 & 502 & 0/5 & 1/5 & 5/5 & 21.2 & 2.7 & 0.0 \\
		90 & $\geq$3600 & $\geq$3600 & 802 & 0/5 & 0/5 & 5/5 & 20.8 & 5.8 & 0.0 \\
		\bottomrule
	\end{tabular}
	\label{table:BnC_SNIP}
\end{table}

\begin{table}[tb!]
	\centering
	\caption{Exact solution results for LLA instances. Time and gap results are averages over six instances.}
	\begin{tabular}{cccccccccc}
		\toprule
		$p$ & \multicolumn{3}{c}{Avg soln time (s)} & \multicolumn{3}{c}{\# solved instances} & \multicolumn{3}{c}{Avg opt gap (\%)}\\
		\cmidrule(lr){2-4}\cmidrule(lr){5-7}\cmidrule(lr){8-10}
		& EXT & BBC & IBC & EXT & BBC & IBC & EXT & BBC & IBC \\
		\midrule\midrule\addlinespace
		12 & $\geq$1803 & $\geq$3100 & $\geq$715 & 4/6 & 2/6 & 5/6 & 0.2 & 0.5 & $<$0.1 \\
		16 & $\geq$3600 & $\geq$3600 & $\geq$3600 & 0/6 & 0/6 & 0/6 & 2.0 & 3.7 & \phantom{$<$}0.3 \\
		20 & $\geq$3600 & $\geq$3600 & $\geq$3600 & 0/6 & 0/6 & 0/6 & 2.1 & 4.9 & \phantom{$<$}0.4\\
		\bottomrule
	\end{tabular}
	\label{table:BnC_LLA}
\end{table}

\begin{table}[tb!]
	\centering
	\caption{Exact solution results for CAP instances. Time and gap results are averages of four instances.}
	\begin{tabular}{cccccccccc}
		\toprule
		Instance \# & \multicolumn{3}{c}{Avg soln time (s)} & \multicolumn{3}{c}{\# solved instances} & \multicolumn{3}{c}{Avg opt gap (\%)}\\
		\cmidrule(lr){2-4}\cmidrule(lr){5-7}\cmidrule(lr){8-10}
		& EXT & BBC & IBC & EXT & BBC & IBC & EXT & BBC & IBC \\
		\midrule\midrule\addlinespace
		101-104 & \phantom{$\geq$00}78 & $\geq$3600 & $\geq$3600 & 4/4 & 0/4 & 0/4 & \phantom{$<$}0.0 & 14.0 & \phantom{0}9.5\\
		111-114 & \phantom{$\geq$}1006 & $\geq$3600 & $\geq$3600 & 4/4 & 0/4 & 0/4 & \phantom{$<$}0.0 & 10.2 & \phantom{0}8.4\\
		121-124 & $\geq$1432 & $\geq$3600 & $\geq$3600 & 3/4 & 0/4 & 0/4 & $<$0.1 & 18.2 & 17.5\\
		131-134 & \phantom{$\geq$0}667 & $\geq$3600 & $\geq$3600 & 4/4 & 0/4 & 0/4 & \phantom{$<$}0.0 & 22.4 & 17.2\\
		\bottomrule
	\end{tabular}
	\label{table:BnC_CAP}
\end{table}

We summarize the average solution time, number of solved instances and average ending optimality gap obtained by
different methods on the SNIP instances, LLA instances, and CAP instances in Tables \ref{table:BnC_SNIP},
\ref{table:BnC_LLA}, and \ref{table:BnC_CAP}, respectively. For the SNIP and LLA instances, we find that IBC either solves more instances and in less
time, or yields smaller ending optimality gap than EXT and BBC, illustrating that the $I$-sparse cuts can indeed lead to
improvements when solving these instances to optimality. EXT is the most effective method for solving the CAP instances, as observed
also in \cite{bodur2017strengthened}. However, in terms of Benders branch-and-cut based methods, the ending optimality
gap of IBC is modestly smaller than BBC on these instances.

\section{Extension: $I$-sparse tilting}\label{sec:tilting}
We now describe an adaptation of the ideas used to create $I$-sparse inequalities to a more general setting by tilting a
given valid inequality. Let $D\subseteq\{0,1\}^n\times\bR^{p}$ denote a mixed-binary set and let $\bx$ and $\by$ denote the
associated binary and continuous variables in the description of $D$, respectively. Let $D^R$ be a relaxation of $D$
with $D^R\cap(\{0,1\}^n\times\bR^{p})=D$. 
We assume we are given a ``base'' inequality $\alpha^T\bx+\beta^T\by\leq \gamma$ that is valid for $D$, and propose a method
for ``tilting'' this inequality by modifying a subset of coefficients $\alpha_I$ of the binary variables, for $I
\subseteq [p]$.
Note that the $I$-sparse inequalities fit this more general setting by setting $p=1$, $y_1 = \theta$, and using 
$-\theta\leq-\text{LB}$ as the base inequality, where LB is a lower bound on $Q(\bx)$ over $\bx \in X$.

Consider a nonempty $I\subseteq [p]$ and the multi-term 0-1 disjunctive relaxation $\DRI$ of $D$ defined as\begin{displaymath}
	\DRI:=\bigcup_{\bchi\in\{0,1\}^I}\DRIchi,
\end{displaymath}
where $\DRIchi:=\{(\bx,\by) \in D^R:\bx_I=\bchi\}$. We provide a sufficient condition for changes to the coefficients
$\alpha_I$ in the inequality $\alpha^T\bx+\beta^T\by\leq \gamma$ that assures the tilted inequality is valid for $D$.
\begin{prop}\label{prop:tilting}
	Let $\alpha^T\bx+\beta^T\by\leq\gamma$ be a valid inequality for $D$.  Then an inequality $\mu^T\bx+\beta^T\by\leq \eta$ with
	$\mu_{[n]\setminus I}=\alpha_{[n]\setminus I}$ is valid for $D$ if\begin{equation}\label{condt:valid_tilting}
		\eta - \mu^T_{I}\chi \geq\min\{\bnuIab,\gamma - \alpha^T_I\bchi\},\quad\forall \bchi\in\{0,1\}^I,
	\end{equation}
	where $\bnuIab=\max\{\alpha^T_{[n]\setminus I}x_{[n]\setminus I}+\beta^Ty:(\bx,\by)\in \DRIchi \}$.
\end{prop}
\begin{proof}
	Note that $\mu^T\bx+\beta^T \by\leq \eta$ with $\mu_{[n]\setminus I}=\alpha_{[n]\setminus I}$ is valid for $D$ if and
	only if it is valid for $\{(\bx,\by)\in D:\bx_I=\bchi \}$ for all $\bchi\in\{0,1\}^I$, i.e.,
	\begin{equation}\label{tilt_iff}
		\max_{(\bx,\by)\in D}\{\mu^T\bx+\beta^T\by :\bx_I=\bchi\}\leq \eta,~\quad\forall\bchi\in\{0,1\}^I.
	\end{equation}
	Since $\alpha^T\bx+\beta^T\by\leq\gamma$ is valid for $D$, $\beta^T \by \leq -\alpha^T \bx + \gamma$ for all $(\bx,\by) \in D$,
	and hence for all $\bchi\in\{0,1\}^I$ we have
	\begin{align*}
		\max_{(\bx,\by)\in D}\{\mu^T\bx+\beta^T\by : \bx_I=\bchi\}
		&\leq\max_{(\bx,\by) \in D} \{ \mu^T \bx - \alpha^T \bx + \gamma : \bx_I = \bchi\} \\
		&=\max_{(\bx,\by)\in D}\{(\mu_I-\alpha_I)^T\bx_I+\gamma :\bx_I=\bchi\} \\
		&=(\mu_I-\alpha_I)^T\bchi+\gamma.
	\end{align*}
	On the other hand, for all $\bchi\in\{0,1\}^I$,\begin{align*}
		\max_{(\bx,\by)\in D}\{\mu^T \bx+\beta^T \by :\bx_I=\bchi\}
		&\leq\max_{(\bx,\by)\in \DRIchi}\{\mu^T \bx+\beta^T \by\} \\
		& =\bnuIab+\mu_{I}^T\bchi.
	\end{align*}
	Therefore, 
	\begin{align}
		\max_{(\bx,\by)\in D}\{\mu^T \bx+\beta^T \by : \bx_I=\bchi\} &\leq
		\min\{\bnuIab+\mu_{I}^T\bchi,(\mu_I-\alpha_I)^T\bchi+\gamma\} \nonumber \\
		& =\min\{\bnuIab,\gamma-\alpha_I^T\bchi \}+\mu_I^T\bchi.
		\label{eq:sub}
	\end{align}
	Replacing $\max_{(\bx,\by)\in D}\{\mu^T \bx+\beta^T \by : \bx_I=\bchi\}$ in \eqref{tilt_iff} by the upper bound in \eqref{eq:sub}
	yields the condition \eqref{condt:valid_tilting}.
	\qed
\end{proof}

Note that $\alpha^T \bx+\beta^T \by\leq\gamma$ does not need to be valid for $\DRI$.  But if $\alpha^T \bx+\beta^T
\by\leq\gamma$ is valid for $\DRI$, then we can simply replace the maximum in \eqref{condt:valid_tilting} by $\bnuIab$
since $\bnuIab\geq\alpha_I^T\bchi-\gamma$ for all $\bchi\in\{0,1\}^I$ in that case. Similar to \eqref{CGLP}, given a
candidate solution $(\hat{\bx},\hat{\by})$ and fixed $I$, we can write down a cut generating linear program for separating from the inequalities satisfying \eqref{condt:valid_tilting}: \begin{displaymath}
	\max\Big\{\sum_{i\in I}\mu_i\hat{x}_i-\eta:\eta - \mu_I^T\bchi\geq\min\{\bnuIab,\gamma - \alpha_I^T\bchi \},~\bchi\in\{0,1\}^I \Big\}.
\end{displaymath}
We call any valid inequality $\mu^T\bx+\beta^T\by\leq\eta$ satisfying \eqref{condt:valid_tilting} an $I$-sparse tilting of
the inequality $\alpha^T\bx+\beta^T \by\leq\gamma$.  One may also iteratively tilt an inequality by applying an $I_k$-sparse tilting sequentially for some sequence $I_1,I_2,\ldots\subseteq[p]$.

We next provide an example of obtaining valid inequalities using $I$-sparse tilting by showing that perspective cuts \cite{frangioni2006perspective} for convex functions with indicator variables can be obtained by $I$-sparse tilting of subgradient inequalities for some $I$ with $|I|=1$.

\begin{exmp}[Perspective cuts]
	The following mixed-binary structure is common in many applications with on/off decisions (e.g., the quadratic uncapacitated facility location problem \cite{gunluk2007minlp}):\begin{displaymath}
		D=\{(x,\bz,\theta)\in \{0,1\}\times\bR^{m+1}:\bzero \leq \bz\leq \bu x,~\theta\geq f(\bz) \},
	\end{displaymath}
	where $\bu\in\bR_+^m$ and $f:\bR^m_+\rightarrow\bR$ is a closed convex function with $f(\bzero)=0$. The convex hull of $F$ can be characterized by the perspective function of the function $f$ \cite{ceria1999convex}. One way of obtaining $\conv(F)$ by the perspective function is to add the (potentially infinitely many) perspective cuts \cite{frangioni2006perspective} of the form:\begin{equation}\label{cut:perspective}
		\theta\geq \bs^T \bz+(f(\bar{\bz})-\bs^T\bar{\bz})x,
	\end{equation}
	where $\bar{\bz}\in[\bzero,\bu]$ and $\bs\in\partial f(\bar{\bz})$. We show that \eqref{cut:perspective} can be obtained by $I$-sparse tilting of the following subgradient inequality of $f$:\begin{equation}\label{cut:subgradient}
		\theta\geq f(\bar{\bz})+\bs^T(\bz-\bar{\bz})\Leftrightarrow 0x+\bs^T\bz-\theta\leq-f(\bar{\bz})+\bs^T\bar{\bz}.
	\end{equation}
	By applying $I$-sparse tilting (Proposition \ref{prop:tilting}) with $I=\{1\}$, the tilted inequality $\mu x+\bs^T\bz-\theta\leq\eta$ is valid if the following two inequalities hold:\begin{enumerate}
		\item $\eta\geq \min\{0,\bs^T\bar{\bz}-f(\bar{\bz}) \}=0$ (because $0=f(\bzero)\geq
		f(\bar{\bz})+\bs^T(\bzero-\bar{\bz})$);
		\item $\eta - \mu \geq \min\big\{\max\{\bs^T\bz -\theta:\theta\geq f(\bz),\bzero\leq \bz\leq
		\bu\},\bs^T\bar{\bz}-f(\bar{\bz})
		\big\}=\bs^T\bar{\bz}-f(\bar{\bz})$ (because \eqref{cut:subgradient} attains equality at
		$(\bz,\theta)=(\bar{\bz},f(\bar{\bz}))$).
	\end{enumerate}
	We then obtain the perspective cut \eqref{cut:perspective} by choosing $(\mu,\eta)$ with both inequalities satisfied
	at equality, i.e., by setting $\mu=f(\bar{\bz})-\bs^T\bar{\bz}$ and $\eta=0$.
\end{exmp}

\section{Conclusion}\label{sec:conc}
We investigate methods for generating $I$-sparse cuts for the epigraph of a function of binary variables. Two selection rules are proposed to choose the support $I$. Numerical experiments demonstrate that $I$-sparse cuts are very effective on problems with sparse features.

We also extend our idea to strengthen valid inequalities by tilting coefficients on a sparse subset of variables. It would be interesting to explore this idea further. Another direction of future work is to integrate our techniques into SDDiP \cite{zou2019stochastic} to solve multistage stochastic integer programs.



\bibliographystyle{splncs04}
\bibliography{refs}   

\appendix
\section*{Appendix}
\subsection*{SNIP problem}
The SNIP problem \cite{pan2008minimizing} is a two-stage stochastic integer program with pure binary first-stage and
continuous second-stage variables. In this problem, by installing sensors on some arcs of a directed network to in the
first stage, the defender tries to find the attacker and minimize the probability that the attacker travels from the
origin to the destination undetected. In the second stage, the origin and destination of the attacker  are observed and
the attacker chooses to travel on the maximum reliability path from its origin to its destination. Let $N$ and $A$
denote the node set and the arc set of the network and let $D\subseteq A$ denote the set of interdictable arcs. The
first-stage variables are denoted by $\bx$, where $x_a=1$ if and only if the defender installs a sensor on arc $a\in D$. Each scenario $s \in S$ is associated with a possible origin/destination combination of the attacker, with $u^s$ representing the origin and $v^s$ representing the destination of the attacker for each scenario $s \in S$. The second-stage variables are denoted by $\pi$, where $\pi^s_i$ denotes the maximum probability of reaching destination $v^s$ undetected from node $i$ in scenario $s$. The budget for installing sensors is $b$, and the cost of installing a
sensor on arc $a$ is $c_a$ for each arc $a\in D$. For each arc $a\in A$, the probability of traveling on arc $a$ undetected is $r_{a}$ if the arc is not interdicted, or $q_a$ if the arc is interdicted. Parameter $\bar{\pi}_j^s$ denotes the maximum probability of reaching the destination undetected from node $j$ when no sensors are installed. The extensive formulation of the problem is as follows:\begin{displaymath}
	\begin{aligned}
		\min_{\bx,\bpi^s}\ &\sum_{s\in S}p_s\pi_{u^s}^s\\
		\text{s.t. }\ &\sum_{a\in A}c_ax_a\leq b,\\
		&\pi_i^s-r_a\pi_j^s\geq 0, &&a=(i,j)\in A\setminus D,\ s\in S,\\
		&\pi_i^s-r_a\pi_j^s\geq -(r_a-q_a)\bar{\pi}_j^sx_a, &&a=(i,j)\in D,\ s\in S,\\
		&\pi_i^s-q_a\pi_j^s\geq 0, &&a=(i,j)\in D,\ s\in S,\\
		&\pi_{v^s}^s=1, &&s\in S,\\
		&x_a\in \{0,1\}, &&a\in D.
	\end{aligned}
\end{displaymath}
We use the SNIP instances from \cite{pan2008minimizing}. We consider instances with snipno = 3, and budget $b\in\{30,50,70,90\}$. All instances have 320 first-stage binary variables, 2586 second-stage continuous variables per scenario and 456 scenarios.

\subsection*{LLA problem}
The LLA problem is introduced in \cite{mendez2014branch}. In this problem, a retailer chooses a set of items to display for customers to purchase. The model assumes all customers are from a set of customer segments. For each customer segment $k$, the customers arrive according to a Poisson process with rate $\lambda_k$ and only purchases products in the consideration set $C_k$. For product $i$, the relative attractiveness of it to customers in segment $k$ is $v_i^k$, and the retailer earns a positive profit $w_i$ for each purchase of it by the customers. The preference of not purchasing anything is denoted by $v_0^k$ for customers in segment $k$. The retailer can at most choose $c\cdot n$ of the items to display. The problem can be formulated as a MINLP as follows:\begin{displaymath}
	\max_{\bx}\Big\{\sum_{k=1}^{N}\sum_{i\in C_k}\frac{\lambda_kv_i^kw_ix_i}{\sum_{i\in
			C_k}v_i^kx_i+v_0^k}:\sum_{i=1}^nx_i\leq c\cdot n, \bx\in\{0,1\}^n \Big\}.
\end{displaymath}
In \cite{mendez2014branch}, the authors reformulate the MINLP as a MILP by introducing a variable $y_k$ to represent the value of $1/(\sum_{i\in C_k}v_i^kx_i+v_0^k)$ and a variable $z_{ik}$ to linearize the product $x_iy_k$. The MILP reformulation of the problem is as follows:\begin{align*}
	\max_{\bx,\by,\bz}~&\sum_{k=1}^K\sum_{i\in C_k}\lambda_kv_i^kw_i z_{ik}\\
	\text{s.t. }&v_0^ky_k+\sum_{i\in C_k}v_i^kz_{ik}=1,&&\hspace{-2cm}k\in[K],\\
	&v_0^ky_k-v_0^kz_{ik}\leq 1-x_i,&&\hspace{-2cm}k\in[K],i\in C_k,\\
	&z_{ik}\leq y_k,&&\hspace{-2cm}k\in[K],i\in C_k,\\
	&(v_0^k+v_i^k)z_{ik}\leq x_i,&&\hspace{-2cm}k\in[K],i\in C_k,\\
	&\sum_{i=1}^n x_i\leq c\cdot n,\\
	&\bx\in\{0,1\}^n,~\by\geq \bzero,~\bz\geq \bzero.
\end{align*}
For the generation of test instances, we follow the basic scheme of generating type-1 problems in \cite{mendez2014branch}, but increase the value of some of the parameters to make the instances harder. We set $n=500$ and $N=200$. For each segment $k$, its arrival rate $\lambda_k$ and the preference of no purchase $v_0^k$ are randomly generated according to the uniform distributions Uniform$([0,1])$ and Uniform$([0,4])$, respectively. The preference $v_i^k$ for segment $k$ of purchasing product $i$ is randomly generated according to the discrete uniform distribution Uniform$(\{0,1,\ldots,10\})$. For each odd $k\in K$, $C_k$ is a independently randomly chosen subset of $[n]$ with size $p\in\{12,16,20\}$. For each even $k\in [N]$, $C_k$ is a random subset of $C_{k-1}$ of size $p/2$. The profit $w_i$ of product $i$ is independently randomly generated according to the uniform distribution Uniform$([100,U])$ with $U\in\{150,350\}$. The capacity parameter $c$ is chosen from $\{20\%,50\%,100\%\}$. We generate one instance for each combination of $(p,U,c)$.

\subsection*{CAP problem}
The stochastic CAP problem \cite{bodur2017strengthened} is a generalization of the deterministic CAP problem
\cite{louveaux1986discrete}, which can be formulated as a stochastic two-stage integer program. In this problem, the
decision maker chooses to open a set of facilities to meet uncertain customer demands. The first-stage variables are
denoted by $\bx$ with $x_i=1$ if and only if facility $i$ is chosen to be opened. The second-stage variables are denoted
by $\by$, where $y_{ij}^k$ is the amount of the $j$th customer's demand met by facility $i$ in scenario $k$. For each facility $i$, the associated opening cost and its capacity are denoted by $f_i$ and $s_i$, respectively. The cost associated with satisfying a unit of the $j$th customer demand using facility $i$ (sending a unit of flow from facility $i$ to customer $j$) is denoted by $q_{ij}$. The $j$th customer's demand under scenario $k$ is denoted by $\lambda_j^k$. The extensive formulation of the problem is as follows:\begin{align*}
	\min_{\bx,\by}~ &\sum_{i=1}^nf_ix_i+N^{-1}\sum_{k=1}^N&&\hspace{-3.25cm}\sum_{i=1}^n\sum_{j=1}^m q_{ij}y_{ij}^k\\
	\text{s.t. } & \sum_{i=1}^ny_{ij}^k\geq\lambda_j^k, &&\hspace{-3.3cm}j\in[m],~k\in[K],\\
	&\sum_{j=1}^my_{ij}^k\leq s_ix_i, &&\hspace{-3.3cm}i\in[n],~k\in[K],\\
	&\sum_{i=1}^ns_ix_i\geq\max_{k\in[K]}\sum_{j=1}^m\lambda_j^k,\\
	&\bx\in\{0,1\}^n,~\by\geq \bzero.
\end{align*}
There are in total 16 CAP test instances all taken from \cite{bodur2017strengthened} with $n\in\{25,50\}$, $m=50$ and $N=250$.


%
%

	
\end{document}